\documentclass[11pt]{amsart}
\usepackage{latexsym,amscd,amssymb, graphicx, color, amsthm, bm, amsmath, cancel, enumitem, ytableau}  
\usepackage{tikz}
\usepackage[margin=1in]{geometry}
\usepackage{hyperref}
\usepackage[all,cmtip]{xy}
\usepackage{mathtools}
\usepackage{rotating}
\usetikzlibrary{math, arrows.meta}
\usetikzlibrary{calc}
\usepackage[backend=bibtex]{biblatex}
\numberwithin{equation}{section}

\newtheorem{theorem}{Theorem}[section]
\newtheorem{proposition}[theorem]{Proposition}
\newtheorem{corollary}[theorem]{Corollary}
\newtheorem{lemma}[theorem]{Lemma}

\newtheorem{problem}[theorem]{Problem}
\newtheorem{example}[theorem]{Example}
\newtheorem{remark}[theorem]{Remark}
\newtheorem{definition}[theorem]{Definition}

\theoremstyle{definition}

\newcommand{\Hilb}{{\mathrm{Hilb}}}

\newcommand{\gl}{\mathrm{GL}}
\newcommand{\Frob}{{\mathrm{Frob}}}
\newcommand{\symm}{{\mathfrak{S}}}
\newcommand{\II}{{\mathbf{I}}}
\newcommand{\gr}{{\mathrm {gr}}}

\newcommand{\lp}[1]{\mathcal{LP}(#1)}
\newcommand{\wid}[1]{\mathrm{wid}(#1)}
\newcommand{\hori}[1]{\mathrm{HS}(#1)}
\newcommand{\horipositive}[1]{\mathrm{PHS}(#1)}
\newcommand{\horiwid}[1]{\mathrm{WHS}(#1)}
\newcommand{\leftshadow}[1]{\mathcal{LS}_{#1}}
\newcommand{\rightshadow}[1]{\mathcal{RS}_{#1}}

\newcommand{\grFrob}{{\mathrm{grFrob}}}

\newcommand{\ZZZ}{{\mathcal{Z}}}

\newcommand{\MMM}{{\mathcal{M}}}

\newcommand{\CC}{{\mathbb{C}}}

\newcommand{\ZZ}{{\mathbb{Z}}}

\newcommand{\Mat}{{\mathrm{Mat}}}

\newcommand{\SYT}{\mathrm{SYT}}
\newcommand{\SSYT}{\mathrm{SSYT}}

\newcommand{\xxx}{{\mathbf{x}}}

\title{Positive combinatorial formulae for involution matrix loci and orbit harmonics}
\author{Hai Zhu}
\address{Department of Mathematics, UC San Diego, La Jolla, CA, 92093, USA}
\email{haz138@ucsd.edu}
\date{\today}

\begin{document}

\begin{abstract}
    Let $\MMM_{n,a}$ be the set consisting of involutions in the symmetric group $\symm_n$ with exactly $a$ fixed points, and apply the orbit harmonics method to obtain a graded $\symm_n$-module $R(\MMM_{n,a})$. Liu, Ma, Rhoades, and Zhu figured out a signed combinatorial formula for the graded Frobenius image $\grFrob(R(\MMM_{n,a});q)$ of $R(\MMM_{n,a})$. Our goal is to cancel these signs. Finally, we find two positive combinatorial formulae for $\grFrob(R(\MMM_{n,a});q)$. As an application, we deduce a series of $\symm_n$-equivariant isomorphisms between graded components $R(\MMM_{n,a})_d$ and $R(\MMM_{n,a^\prime})_d$ for some integers $a\neq a^\prime$ and $d$. Our positive formulae also yield potential attempts to find a linear basis for $R(\MMM_{n,a})$ and a statistic $\mathrm{stat}:\MMM_{n,a}\rightarrow\mathbb{Z}_{\ge0}$ to interpret the Hilbert series $\Hilb(R(\MMM_{n,a});q)$ of $R(\MMM_{n,a})$.
\end{abstract}

\maketitle

\section{Introduction}\label{sec:intro}

Orbit harmonics is a vital method in combinatorial representation theory. Input a finite locus $\mathcal{Z}\subseteq\CC^N$ and apply the orbit harmonics method to $\ZZZ$. Then the output is $R(\ZZZ)$, a quotient ring of $\CC[\xxx_N]$ by a graded ideal. If $\ZZZ$ further carries an action of some subgroup $G\le\gl(\CC^N)$, then we have an isomorphism of $G$-modules $R(\ZZZ)\cong\CC[\ZZZ]$ where $\CC[\ZZZ]$ is the space of all functions $f:\ZZZ\rightarrow\CC$. That is, $R(\ZZZ)$ is a graded refinement of the above-mentioned action $G\curvearrowright\ZZZ$. Furthermore, $R(\ZZZ)$ provides algebraic tools to understand combinatorics on $\ZZZ$.

Orbit harmonics interacts with many fields in mathematics, such as cohomology theory \cite{garsia1992certain}, Macdonald theory \cite{griffin2021ordered,haglund2018ordered}, cyclic sieving \cite{oh2022cyclic}, Donaldson–Thomas theory \cite{reineke2023zonotopal}, and Ehrhart theory \cite{reiner2024harmonics}.

Rhoades \cite{rhoades2024increasing} initiated the implementation of the orbit harmonics method into matrix loci $\ZZZ\subseteq\Mat_{n\times n}(\CC)$, studying the permutation matrix locus $\ZZZ=\symm_n$ carrying an action of the subgroup $\symm_n\times\symm_n\le\gl(\Mat_{n\times n}(\CC))$. Thanks to this work, Chen's log-concavity conjecture \cite{chen} can be studied from an algebraic perspective. Then Liu \cite{liu2024viennot} extended Rhoades \cite{rhoades2024increasing} to colored permutations.

Permutation matrix loci were revisited by Liu et al. \cite{liu2025involution}. They considered $\ZZZ = \MMM_{n,a}$, the matrix locus consisting of involutions in $\symm_n$ with $a$ fixed points, carrying the conjugate action of $\symm_n$. They found an explicit combinatorial expression of the graded Frobenius image $\grFrob(R(\MMM_{n,a});q)$ of $R(\MMM_{n,a})$ (Theorem~\ref{thm:conjugacy-module-character}). However, this is a signed formula. In algebraic combinatorics, people usually prefer positive combinatorial expressions without minus signs.

We find two positive combinatorial formulae for the graded Frobenius image $\grFrob(R(\MMM_{n,a});q)$. Our main theorem is the second one which is more compatible with the Schur expansion of $\Frob(\CC[\MMM_{n,a}]) = \Frob(R(\MMM_{n,a})) = h_{(n-a)/2}[h_2]\cdot h_a$ using Pieri's rule.

\begin{theorem}\label{thm:main}
    For integers $n,a$ such that $n>0$, $0\le a\le n$, $a\equiv n\mod{2}$, we have that
    \[\grFrob(R(\MMM_{n,a});q) = \sum_{\lambda/\mu} q^{\frac{n+a-\wid{\lambda/\mu}}{2}}\cdot s_\lambda\]
    where the summation is over all the horizontal stripes $\lambda/\mu$ such that:
    \begin{itemize}
        \item $\lambda\vdash n$
        \item $\mu\vdash n-a$ is an even partition
    \end{itemize}
\end{theorem}

In Theorem~\ref{thm:main}, the statistic $\wid{}$ associated to horizontal stripes is given by Definition~\ref{def:width}. We will prove Theorem~\ref{thm:main} in Section~\ref{sec:good-ver}.

\section{Background}\label{sec:background}

\subsection{Orbit harmonics}\label{subsec:orbit-har}

Given a finite locus $\ZZZ\subseteq\CC^N$ carrying an action of a subgroup $G\subseteq\gl(\CC^N)$, we define its \emph{vanishing ideal} $\II(\ZZZ)\subseteq\CC[\xxx_N]$ by
\[\II(\ZZZ)\coloneqq\{f\in\CC[\xxx_N]\,:\,\text{$f(z) = 0$ for all $z\in\ZZZ$}\}\]
where $\xxx_N\coloneqq\{x_1,\cdots,x_N\}$ is the variable set. Multivariate Lagrange interpolation indicates that
$\CC[\ZZZ]\cong\CC[\xxx_N]/\II(\ZZZ)$ as $G$-modules, where $\CC[\ZZZ]$ is the space of all functions $f:\ZZZ\rightarrow\CC$.

Now we construct a graded version of the $G$-modules above. For an ideal $I\subseteq\CC[\xxx_N]$, the \emph{associated graded ideal} $\gr I$ of $I$ is given by \[\gr I\coloneqq\langle f\,:\,\text{$f$ is the highest homogeneous component of some nonzero polynomial $g\in I$}\rangle.\] Then define the \emph{orbit harmonics ring} $R(\ZZZ)$ of $\ZZZ$ by
\[R(\ZZZ)\coloneqq\CC[\xxx_N]/\gr\II(\ZZZ).\]

The orbit harmonics method possesses an interesting and essential property: a chain of $G$-module isomorphisms
\[\CC[\ZZZ]\cong\CC[\xxx_N]/\II(\ZZZ)\cong\CC[\xxx_N]/\gr\II(\ZZZ)=:R(\ZZZ)\] where the first isomorphism arises from multivariate Lagrange interpolation but the second isomorphism is not explicit (i.e. without an explicit $G$-equivariant linear map). In fact, the second isomorphism is an abstract isomorphism arising from the following $G$-equivariant linear maps
\begin{align}\label{eq:grade-isom}\CC[\xxx_N]_d/\gr\II(\ZZZ)_d \overset{\cong}{\longrightarrow}\frac{\CC[\xxx_N]_{\le d}/\II(\ZZZ)\cap\CC[\xxx_N]_{\le d}}{\CC[\xxx_N]_{\le d-1}/\II(\ZZZ)\cap\CC[\xxx_N]_{\le d-1}}\end{align}
induced by
\begin{align*}
    \CC[\xxx_N]_d&\longrightarrow\CC[\xxx_N]_{\le d}/\II(\ZZZ)\cap\CC[\xxx_N]_{\le d} \\
    f &\longmapsto f \mod{\II(\ZZZ)\cap\CC[\xxx_N]_{\le d}}
\end{align*}
which maps $\gr\II(\ZZZ)_d$ into $\CC[\xxx_N]_{\le d-1}/\II(\ZZZ)\cap\CC[\xxx_N]_{\le d-1}$. The Map~\eqref{eq:grade-isom} is bijective since the standard result \cite[Lemma 2.2]{liu2025involution}.

\subsection{Partitions and Young tableaux}\label{subsec:tableau}

%tableau

Given a positive integer $n\in\mathbb{N}$, a \emph{partition} $\lambda$ of $n$ is a weakly decreasing sequence of nonnegative integers $\{\lambda_i\}_{i=1}^\infty$ such that $\sum_{i=1}^\infty\lambda_i = n$. In this case, we write $\lambda\vdash n$. We usually identify partition $\lambda$ with its \emph{Young diagram} obtained by placing $\lambda_i$ boxes in the $i$-th row ($i=1,2,\cdots$). A \emph{standard Young tableau} of shape $\lambda$ is a bijective filling of $[n]$ into the boxes of $\lambda$, such that the entries are rightwards increasing along each row and downwards increasing along each column. A \emph{semi-standard Young tableau} of shape $\lambda$ is a filling of positive integers $\mathbb{N}$ into the boxes of $\lambda$, such that the entries are rightwards weakly increasing along each row and downwards strongly increasing along each column. Write $\SYT(\lambda)\coloneqq\{\text{standard Young tableaux of shape $\lambda$}\}$ and $\SSYT(\lambda)\coloneqq\{\text{semi-standard Young tableaux of shape $\lambda$}\}$. For a semi-standard Young tableau $P\in\SSYT(\lambda)$, we write $\mathrm{sh}(P) = \lambda$, which means that $P$ is of shape $\lambda$. For instance, $\lambda = (7,4,4,2)\vdash 17$ is a partition. Its Young diagram and one element in $\SYT(\lambda)$ is shown below.
\begin{center}
    \ydiagram{7,4,4,2}\qquad\begin{ytableau}
        1 & 3 & 7 & 10 & 11 & 15 & 17 \cr
        2 & 4 & 8 & 12 \cr
        5 & 6 & 9 & 13 \cr
        14 & 16
    \end{ytableau}
\end{center}

Robinson-Schensted-Knuth correspondence \cite[Section 4] {fulton1997young} is a one-to-one correspondence $\mathrm{RSK}$ from $\mathbb{N}\times\mathbb{N}$ matrices with entries in $\ZZ_{\ge 0}$ (all but finitely many entries are $0$) to pairs of semi-standard Young tableaux of the same shape. $\mathrm{RSK}$ has several equivalent definitions, among which the most classical one uses the row insertion operation. In particular, the restriction of $\mathrm{RSK}$ to permutation matrices yields an elegant bijection
\[\symm_n\overset{1:1}{\longrightarrow}\bigsqcup_{\lambda\vdash n}(\SYT(\lambda)\times\SYT(\lambda)).\]

We will use two standard results for $\mathrm{RSK}$ in Section~\ref{sec:future}. See \cite[Section 4.2, Exercise 4]{fulton1997young} for details.

\begin{proposition}\label{prop:sym-rsk}
    Let $A$ be an $n\times n$ matrix with nonnegative integer entries, and write $\mathrm{RSK}(A)=(P,Q)$. Then $P=Q$ if and only if $A$ is symmetric.
\end{proposition}

\begin{proposition}\label{prop:sym-rsk-odd}
    Let $A$ be a symmetric $n\times n$ matrix with nonnegative integer entries. Then $\mathrm{RSK}(A) = (P,P)$ where the number of odd columns of $\mathrm{sh}(P)$ equals the trace of $A$.
\end{proposition}

\subsection{Symmetric functions}\label{subsec:sym-func}

%horizontal stripe, Pirie's rule, plethsym, restric
Write $\Lambda = \bigoplus_{n\ge 0}\Lambda_n$ for the graded algebra of symmetric functions in an infinite variable set $\xxx=\{x_1,x_2,\cdots\}$ with coefficient field $\CC(q)$. Each graded component $\Lambda_n$ has two linear basis: the complete homogeneous functions $\{h_\lambda\}_{\lambda\vdash n}$ and the Schur functions $\{s_\lambda\}_{\lambda\vdash n}$. The definitions of both can be found in \cite{macdonald1998symmetric}.

For $f=\sum_{\lambda}c_\lambda(q)\cdot s_\lambda\in\Lambda$ where $c_\lambda(q)\in\CC(q)$ and a condition $P$ on partitions, we introduce the truncation notation
\[\{f\}_P\coloneqq\sum_{\text{$\lambda$ satisfies $P$}}c_\lambda(q)\cdot s_\lambda.\]

The multiplication of Schur functions by complete homogeneous functions satisfies Pieri's rule: For any partition $\mu$ and any integer $a\ge 0$, we have that
\[s_\mu\cdot h_a = \sum_{\lambda/\mu}s_\lambda\] summing over horizontal stripes $\lambda/\mu$ of size $a$. Here, a \emph{horizontal stripe} $\lambda/\mu$ is a pair of partitions $\mu\subseteq\lambda$ such that the complement diagram $\lambda\setminus\mu$ has at most one box in each column. The size $\lvert\lambda/\mu\rvert$ of $\lambda/\mu$ is given by $\lvert\lambda/\mu\rvert\coloneqq\lvert\lambda\rvert-\lvert\mu\rvert$. Interestingly, we can deduce the above-stated Pieri's rule from its tableau version \cite[Lemma 4.11]{nelsen2003kostka}:

\begin{proposition}[the tableau version of Pieri's rule]\label{prop:tab-pieri}
    Let $\mu$ be a partition and $a\ge0$ be an integer. Then we have the bijection
    \[\SSYT(\mu)\times\SSYT(a)\overset{1:1}{\longrightarrow}\bigsqcup_{\substack{\text{horizontal stripe $\lambda/\mu$}\\\lvert\lambda/\mu\rvert = a}}\SSYT(\lambda)\]
    where each pair $(P,Q)\in\SSYT(\mu)\times\SSYT(a)$ is mapped to the tableau $R$ obtained by inserting all entries of $Q$ (one by one and from left to right) into $P$ using the row insertion algorithm.
\end{proposition}

In addition to multiplication, $\Lambda$ carries an operation called \emph{plethysm}
\begin{align*}
    \Lambda\times\Lambda &\longrightarrow \Lambda \\
    (f,g) &\longmapsto f[g].
\end{align*}
A standard property of plethysm is that: For any integers $d\ge 0$,
\[h_d[h_2] = \sum_{\substack{\lambda\vdash 2d\\ \text{$\lambda$ is even}}}s_\lambda\] where ``$\lambda$ is even'' means that all parts $\lambda_i$ of $\lambda$ is even. See \cite{macdonald1998symmetric} for more details about plethysm.

\subsection{Representation theory of symmetric groups}\label{subsec:sym-gp}
%Specht,dim
Let $\symm_n$ be the symmetric group of $[n]$. All irreducible representations of $\symm_n$ are classified by \emph{Specht modules} $\{V^\lambda\}_{\lambda\vdash n}$ (see \cite[Section 7]{fulton1997young} for details). $V^\lambda$ satisfies a standard property
\[\dim V^\lambda = \lvert\SYT(\lambda)\rvert.\]
%frob,grfrob

Consider a finite-dimensional $\symm_n$-module $V$. Suppose that $V\cong\bigoplus_{\lambda\vdash n}c_\lambda V^\lambda$ for some $c_\lambda\in\ZZ_{\ge 0}$. We can describe the module structure of $V$ using its \emph{Frobenius image} $\Frob(V)\in\Lambda_n$ given by
\[\Frob(V)\coloneqq\sum_{\lambda\vdash n}c_\lambda\cdot s_\lambda.\] Furthermore, if $V=\bigoplus_{d=0}^m V_d$ is a graded $\symm_n$-module, its graded module structure can be described using its \emph{graded Frobenius image} $\grFrob(V;q)\in\Lambda_n$ given by
\[\grFrob(V;q)\coloneqq\sum_{d=0}^m q^d \cdot \Frob(V_d).\]

\subsection{Basic settings for our main result}\label{subsec:(background)basic-setting}

%def of \MMM_{n,a}
Let $\MMM_{n,a}\coloneqq\{w\in\symm_n\,:\,\text{$w^2=1$ and $w$ has exactly $a$}$\\$\text{fixed points}\}$ and identify $w\in\MMM_{n,a}$ with its permutation matrix. Therefore, $\MMM_{n,a}\subseteq\Mat_{n\times n}(\CC)$ is a finite matrix locus carrying the conjugate action of $\symm_n$ given by $g(w) = gwg^{-1}$.
%variable set x_{n\times n}

Applying the orbit harmonics method to $\MMM_{n,a}$, we obtain a graded $\symm_n$-module
\[R(\MMM_{n,a}) = \CC[\xxx_{n\times n}]/\gr\II(\MMM_{n,a})\]
where the variable set is $\xxx_{n\times n}\coloneqq\{x_{i,j}\,:\,i,j\in[n]\}$. Liu, Ma, Rhoades, and Zhu figured out a signed combinatorial formula \cite[Theorem 5.21]{liu2025involution} for $\grFrob(R(\MMM_{n,a});q)$, which will play an essential role in Section~\ref{sec:bad-ver}:

\begin{theorem}
    \label{thm:conjugacy-module-character}
    Suppose $a \equiv n \mod 2$. The graded Frobenius image of $R(\MMM_{n,a})$ is given by
    \begin{equation}\label{eq:conjugacy-module-character}
        \grFrob(R(\MMM_{n,a});q) = \sum_{d \, = \, 0}^{(n-a)/2} \{
            h_d[h_2] \cdot h_{n-2d} - h_{d-1}[h_2] \cdot h_{n-2d+2}
        \}_{\lambda_1 \leq n-2d+a} \cdot q^d
    \end{equation}  
    where we interpret $h_{-1} := 0$.
\end{theorem}

To study the index set of the Schur expansion of the right-hand side of Equation~\eqref{eq:conjugacy-module-character} using Pieri's rule, we need the following constructions.

\begin{definition}\label{def:path}
    Given a horizontal stripe $\lambda/\mu$, we associate a lattice path $\lp{\lambda/\mu}$ of length $\infty$ with it sequentially as follows:
    \begin{itemize}
        \item Let $\lp{\lambda/\mu}$ start at the origin $(0,0)$.
        \item Scan all the columns of $\lambda$ rightwards. If one column intersects $\lambda/\mu$, append an NE step (i.e. the vector $(1,1)$) to $\lp{\lambda/\mu}$. Otherwise, append an $SE$ step (i.e. the vector $(1,-1)$) to $\lp{\lambda/\mu}$.
        \item Append infinitely many $SE$ steps to $\lp{\lambda/\mu}$.
    \end{itemize}
\end{definition}

\begin{example}\label{ex:path}
    Consider $\lambda=(10,9,6,4,4,3)$ and $\mu=(10,6,4,4,4,2)$. That is, 
    \[\lambda/\mu =  \ydiagram[*(white) \bullet]
 {10+0,6+3,4+2,4+0,4+0,2+1}
 *[*(white)]{10,9,6,4,4,3}.\]
 Then the lattice path $\lp{\lambda/\mu}$ is shown as follows.

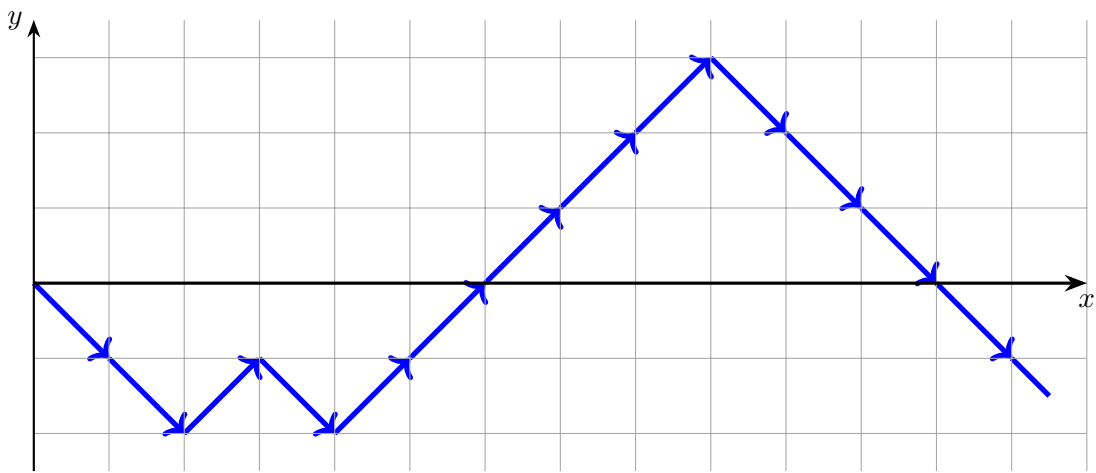
\def\sequence{-1,-1,1,-1,1,1,1,1,1,-1,-1,-1,-1}
\def\gridscale{1}
\begin{tikzpicture}[scale=\gridscale]
    \tikzmath{
    integer \x,\y,\miny,\maxy;
    \x=0;\y=0;\miny=0;\maxy=0;
    for \s in \sequence {
    {\draw[->,line width=2pt,color = blue] (\x,\y) -- (\x+1,\y+\s);};
    \x=\x+1;\y=\y+\s;
    if \y>\maxy then {\maxy = \y;};
    if \y<\miny then {\miny = \y;};
    };
    {\draw[line width=2pt,color=blue] (\x,\y) -- (\x+0.5,\y-0.5);
    \draw[gray!70, thin] (0,\miny-0.5) grid (\x+1,\maxy+0.5);
    \draw[very thick, -Stealth] (0,0) -- (\x+1,0) node[below] {$x$};
    \draw[thick, -Stealth] (0,\miny-0.5) -- (0,\maxy+0.5) node[left]  {$y$};};
    }
\end{tikzpicture}

Note that we can also reconstruct a horizontal stripe $\lambda/\mu$ if we know $\lambda$ and $\lp{\lambda/\mu}$.
\end{example}

\begin{definition}\label{def:pair}
    For a given lattice path $\mathcal{L}$ starting at $(0,0)$ and consisting of infinitely many NE steps and SE steps, an ordered pair $(i,j)$ of positive integers $0<i<j$ is called a \emph{reflection pair} of $\mathcal{L}$ if $(i,j)$ satisfies all of the following three conditions:
    \begin{itemize}
        \item The $i$-th step of $\mathcal{L}$ (denoted by $\mathcal{L}_i$) is an NE step.
        \item The $j$-th step of $\mathcal{L}$ (denoted by $\mathcal{L}_j$) is an SE step.
        \item The horizontal rightward ray beamed by $\mathcal{L}_i$ touches $\mathcal{L}_j$ without obscured by the lattice path $\mathcal{L}$.
    \end{itemize}
\end{definition}

\begin{example}\label{ex:pair}
    All the reflection pairs of the lattice path $\lp{\lambda/\mu}$ shown in Example~\ref{ex:path} are $(3,4)$, $(5,14)$, $(6,13)$, $(7,12)$, $(8,11)$, and $(9,10)$. See the picture below for details.
    
    \def\sequence{-1,-1,1,-1,1,1,1,1,1,-1,-1,-1,-1,-1}
    \def\gridscale{1}
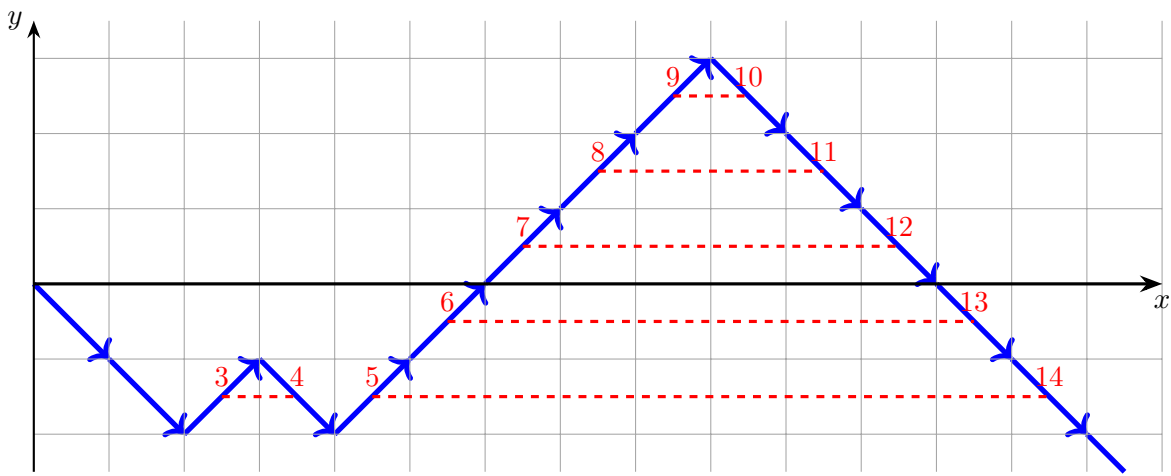
    \begin{tikzpicture}[scale=\gridscale]
    \tikzmath{
    integer \x,\y,\miny,\maxy;
    \x=0;\y=0;\miny=0;\maxy=0;
    for \s in \sequence {
    {\draw[->,line width=2pt,color = blue] (\x,\y) -- (\x+1,\y+\s);};
    \x=\x+1;\y=\y+\s;
    if \y>\maxy then {\maxy = \y;};
    if \y<\miny then {\miny = \y;};
    };
    {\draw[line width=2pt,color=blue] (\x,\y) -- (\x+0.5,\y-0.5);
    \draw[gray!70, thin] (0,\miny-0.5) grid (\x+1,\maxy+0.5);
    \draw[very thick, -Stealth] (0,0) -- (\x+1,0) node[below] {$x$};
    \draw[thick, -Stealth] (0,\miny-0.5) -- (0,\maxy+0.5) node[left]  {$y$};
    \draw[red,dashed,very thick] (2.5,-1.5) -- (3.5,-1.5);
    \node[above,red] at (2.5,-1.5) {$3$};
    \node[above,red] at (3.5,-1.5) {$4$};
    };
    int \j;
    for \i in {5,6,7,8,9} {
    \j=19-\i;
    {\draw[red,dashed,very thick] (\i-0.5,\i-6.5) -- (18.5-\i,\i-6.5);
    \node[above,red] at (\i-0.5,\i-6.5) {$\i$};
    \node[above,red] at (18.5-\i,\i-6.5) {$\j$};};
    };
    }
\end{tikzpicture}

\end{example}

\begin{definition}\label{def:width}
    The \emph{width} $\wid{\lambda/\mu}$ of a horizontal stripe $\lambda/\mu$ is given by
    \[\wid{\lambda/\mu}\coloneqq\max\{\lambda_1,\max\{j\,:\,\text{$(i,j)$ is a reflection pair of $\lp{\lambda/\mu}$ for some $i$}\}\}\]
    by the convention that $\max\varnothing = 0$.
\end{definition}

\begin{remark}\label{rmk:width}
   There is another equivalent way to define $\wid{\lambda/\mu}$. Let $\{c_i\}_{i=1}^{\lambda_1}$ be the sequence given by
   \[c_i\coloneqq\begin{cases}
       -1, &\text{if the $(\lambda_1-i+1)$-th column of $\lambda$ intersects $\lambda/\mu$} \\
       1, &\text{otherwise}
   \end{cases}\]
   and let $M$ be the maximum prefix summation of $\{c_i\}_{i=1}^{\lambda_1}$. Then
   \[\wid{\lambda/\mu} = \lambda_1 + \max\{0,M\}.\]
\end{remark}

\begin{example}\label{ex:width}
    Let $\lambda/\mu$ be the horizontal stripe in Example~\ref{ex:path}. Then Example~\ref{ex:pair} indicates that
    \[\max\{j\,:\,\text{$(i,j)$ is a reflection pair of $\lp{\lambda/\mu}$ for some $i$}\} = 14\]
    and thus \[\wid{\lambda/\mu} = \max\{\lambda_1,14\} = \max\{10,14\} = 14.\]
\end{example}

\section{Bad version: a positive combinatorial formula depending on the graded component}\label{sec:bad-ver}

In this section, we give a positive combinatorial version of Theorem~\ref{thm:conjugacy-module-character} by the cancellation through embedding all the terms of the Schur expansion of $\{h_{d-1}[h_2]\cdot h_{n-2d+2}\}_{\lambda_1\le n-2d+a}$ into the Schur expansion of $\{h_{d}[h_2]\cdot h_{n-2d}\}_{\lambda_1\le n-2d+a}$. That is, we provide a positive combinatorial formula for the Schur expansion of $\{
            h_d[h_2] \cdot h_{n-2d} - h_{d-1}[h_2] \cdot h_{n-2d+2}
        \}_{\lambda_1 \leq n-2d+a}$.

Initially, we need some notations. From now on, \emph{fix} two integers $n$ and $a$ such that $n>0$, $0\le a\le n$ and $a\equiv n \mod{2}$. Let $0\le d\le\frac{n-a}{2}$ be an integer and $\lambda\vdash n$ be a partition such that $\lambda_1\le n-2d+a$. We introduce two sets of horizontal stripes given by:

\begin{align}
\label{eq:def-horizontal-stripe-set}
    \hori{d,\lambda} & \coloneqq \{\lambda/\mu\,:\, \text{$\mu\vdash 2d$ is an even partition such that $\lambda/\mu$ is a horizontal stripe}\} \\
\label{eq:def-positive-horizontal-stripe-set}
    \horipositive{d,\lambda} & \coloneqq \{\lambda/\mu\in\hori{d,\lambda}\,:\,\text{the first $\lambda_1$ steps of $\lp{\lambda/\mu}$ are weakly higher than the $x$-axis}\}.
\end{align}

\begin{example}\label{ex:def-horizontal-stripe-set}
    Recall the horizontal stripe $\lambda/\mu$ in Example~\ref{ex:path} given by $\lambda = (10,9,6,4,4,3)$ and $\mu = (10,6,4,4,4,2)$. Since $\mu\vdash 30$ is an even partition, we have that $\lambda/\mu\in\hori{15,\lambda}$. However, $\lp{\lambda/\mu}$ goes below the $x$-axis at the first step, revealing that $\lambda/\mu\notin\horipositive{d,\lambda}$. 
\end{example}

Further assume that $d>0$. Then we construct a map (in fact a bijection which will be shown in Lemma~\ref{lem:map-NHS-to-HS-bijective})
\begin{align}\label{eq:map-NHS-to-HS}  \Phi_{d,\lambda}:\hori{d,\lambda}\setminus\horipositive{d,\lambda}\longrightarrow\hori{d-1,\lambda}.
\end{align} That is, given a horizontal stripe $\lambda/\mu\in\hori{d,\lambda}\setminus\horipositive{d,\lambda}$, we construct its image $\Phi_{d,\lambda}(\lambda/\mu)\in\hori{d-1,\lambda}$ using the following steps sequentially:
\begin{itemize}
    \item Let $m\ge 0$ be the smallest integer such that $x=m$ is one of the lowest points of the restricted lattice path $\lp{\lambda/\mu}\mid_{0\le x\le\lambda_1}$, i.e. $x=m$ is the first lowest point of $\lp{\lambda/\mu}\mid_{0\le x\le\lambda_1}$.
    \item The lowest points of $\lp{\lambda/\mu}\mid_{0\le x\le\lambda_1}$ is strictly lower than the $x$-axis since $\lambda/\mu\in\hori{d,\lambda}\setminus\horipositive{d,\lambda}$. Consequently, we have that $m>0$. Furthermore, $\mu$ is an even partition, indicating that $m$ is even. Therefore, $m\ge 2$.
    \item As a result of $m\ge 2$, the $(m-1)$-th step and $m$-th step of $\lp{\lambda/\mu}$ exist. Moreover, both of these two steps are SE steps since $x=m$ is the first lowest point of $\lp{\lambda/\mu}\mid_{0\le x\le\lambda_1}$. Therefore, removing the lowest boxes respectively from the $(m-1)$-th column and $m$-th column of $\mu$ (note that both boxes are on the same row) yields a new even partition $\nu\vdash 2d-2$ such that $\lambda/\nu$ is still a horizontal stripe.
    \item Define $\Phi_{d,\lambda}(\lambda/\mu)$ by $\Phi_{d,\lambda}(\lambda/\mu)\coloneqq\lambda/\nu$.
\end{itemize}
\begin{example}\label{ex:map-HHS-to-HS}
    We use a concrete example to illustrate how to obtain $\Phi_{d,\lambda}(\lambda/\mu)$. Consider $\lambda=(17,14,13,8,3,2)$ and $\mu=(14,14,12,6,2)$, yielding a horizontal stripe
    \[\lambda/\mu = \ydiagram[*(white) \bullet]
 {14+3,14+0,12+1,6+2,2+1,0+2}
 *[*(white)]{17,14,13,8,3,2}.\] Then the lattice path $\lp{\lambda/\mu}$ is shown as follows.
 
 \def\sequence{1,1,1,-1,-1,-1,1,1,-1,-1,-1,-1,1,-1,1,1,1,-1,-1,-1,-1}
 \def\gridscale{0.7}
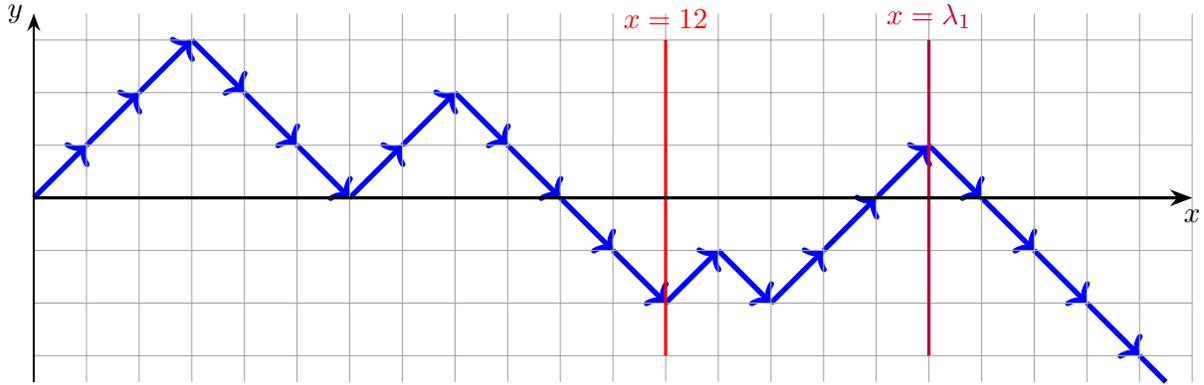
 \begin{tikzpicture}[scale=\gridscale]
    \tikzmath{
    integer \x,\y,\miny,\maxy;
    \x=0;\y=0;\miny=0;\maxy=0;
    for \s in \sequence {
    {\draw[->,line width=2pt,color = blue] (\x,\y) -- (\x+1,\y+\s);};
    \x=\x+1;\y=\y+\s;
    if \y>\maxy then {\maxy = \y;};
    if \y<\miny then {\miny = \y;};
    };
    {\draw[line width=2pt,color=blue] (\x,\y) -- (\x+0.5,\y-0.5);
    \draw[gray!70, thin] (0,\miny-0.5) grid (\x+1,\maxy+0.5);
    \draw[red,very thick] (12,\miny) -- (12,\maxy) node[above] {$x=12$};
    \draw[purple,very thick] (17,\miny) -- (17,\maxy) node[above] {$x=\lambda_1$};
    \draw[very thick, -Stealth] (0,0) -- (\x+1,0) node[below] {$x$};
    \draw[thick, -Stealth] (0,\miny-0.5) -- (0,\maxy+0.5) node[left]  {$y$};};
    }
\end{tikzpicture}
Note that $x=12$ is the first lowest point of $\lp{\lambda/\mu}_{0\le x\le\lambda_1}$ and hence $m=12$. Now we remove the lowest boxes respectively from the $11$-th and $12$-th columns of $\mu$, or, equivalently, add two boxes $(3,11)$ and $(3,12)$ into the horizontal stripe $\lambda/\mu$, obtaining a new horizontal stripe
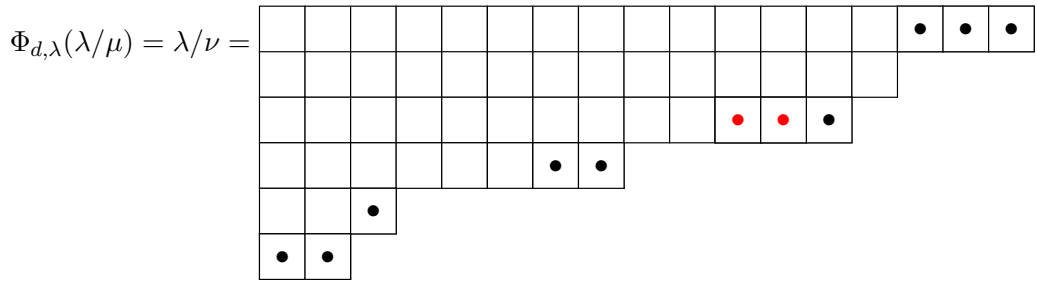
\[\Phi_{d,\lambda}(\lambda/\mu) = \lambda/\nu = \ydiagram[*(white) \textcolor{red}{\bullet}]{17+0,14+0,10+2,8+0,3+0,2+0}*[*(white) \bullet]
 {14+3,14+0,10+3,6+2,2+1,0+2}
 *[*(white)]{17,14,13,8,3,2}.\]
\begin{remark}\label{rmk:turning}It is helpful to understand $\Phi_{d,\lambda}$ as a ``turning upwards" operation on lattice paths, which will be used in the proof of Lemma~\ref{lem:map-NHS-to-HS-bijective}. Intuitively, the lattice path $\lp{\Phi_{d,\lambda}(\lambda/\mu)}$ arises from $\lp{\lambda/\mu}$ by turning the $(m-1)$-th and $m$-th steps upwards as the following figure. We will see that $x=m-2$ is the last lowest point of $\lp{\Phi_{d,\lambda}(\lambda/\mu)}\mid_{0\le x\le \lambda_1}$.

 \def\gridscale{0.7}
 \def\sequence{1,1,1,-1,-1,-1,1,1,-1,-1,1,1,1,-1,1,1,1,-1,-1,-1,-1}
 \def\oldsequence{-1,-1,1,-1,1,1,1,-1,-1,-1,-1}
 \begin{tikzpicture}[scale=\gridscale]
    \tikzmath{
    integer \x,\y,\miny,\maxy;
    \x=0;\y=0;\miny=0;\maxy=0;
    for \s in \sequence {
    {\draw[->,line width=2pt,color = blue] (\x,\y) -- (\x+1,\y+\s);};
    \x=\x+1;\y=\y+\s;
    if \y>\maxy then {\maxy = \y;};
    if \y<\miny then {\miny = \y;};
    };
    {\draw[line width=2pt,color=blue] (\x,\y) -- (\x+0.5,\y-0.5);};
    \x=10; \y=0;
    for \s in \oldsequence {
    {\draw[->,dashed,line width=2pt,color = blue] (\x,\y) -- (\x+1,\y+\s);};
    \x=\x+1;\y=\y+\s;
    if \y>\maxy then {\maxy = \y;};
    if \y<\miny then {\miny = \y;};
    };
    {\draw[dashed,line width=2pt,color=blue] (\x,\y) -- (\x+0.5,\y-0.5);
    \draw[gray!70, thin] (0,\miny-0.5) grid (\x+1,\maxy+0.5);
    \draw[red,very thick] (12,\miny) -- (12,\maxy) node[above] {$x=m$};
    \draw[purple,very thick] (17,\miny) -- (17,\maxy) node[above] {$x=\lambda_1$};
    \draw[very thick, -Stealth] (0,0) -- (\x+1,0) node[below] {$x$};
    \draw[thick, -Stealth] (0,\miny-0.5) -- (0,\maxy+0.5) node[left]  {$y$};};
    }
\end{tikzpicture}
\end{remark}
\end{example}

\begin{lemma}\label{lem:map-NHS-to-HS-bijective}
    The map $\Phi_{d,\lambda}$ in \eqref{eq:map-NHS-to-HS} is bijective.
\end{lemma}

\begin{proof}
Fix $n,a,d,\lambda$ such that $n>0$, $0\le a \le n$, $a\equiv n\mod{2}$, $0<d\le\frac{n-a}{2}$, $\lambda\vdash n$, $\lambda_1 \le n-2d+a$.

\textbf{Injectivity:} For a horzontal stripe $\lambda/\mu\in\hori{d,\lambda}\setminus\horipositive{d,\lambda}$, let $x=m$ be the first lowest point of the restricted lattice path $\lp{\lambda/\mu}\mid_{0\le x\le\lambda_1}$, and let $y=h$ be the height of this point. Write $\Phi_{d,\lambda}(\lambda/\mu) = \lambda/\nu$.
\begin{center}
    \textbf{Claim:} $x=m-2$ is the last lowest point of $\lp{\lambda/\nu}\mid_{0\le x\le\lambda_1}$.
\end{center}
In fact, we had shown that $m\ge 2$ during the construction of $\Phi_{d,\lambda}$. Note that $\lp{\lambda/\nu}$ arises from $\lp{\lambda/\mu}$ by turning the $(m-1)$-th and $m$-th steps upwards (see Remark~\ref{rmk:turning}). This ``turning upwards" operation moves all the nodes of $\lp{\lambda/\mu}\mid_{x\ge m}$ upwards by $4$ units. Consequently, all the nodes of $\lp{\lambda/\nu}\mid_{m\le x\le\lambda_1}$ is weakly higher than $y=h+4$. In addition, the ``turning" operation moves nodes $\lp{\lambda/\mu}\mid_{x=m-1}$ and $\lp{\lambda/\mu}\mid_{x=m}$ strongly above $y=h+2$. Therefore, $\lp{\lambda/\nu}\mid_{m-2<x\le\lambda_1}$ is strongly higher than $y=h+2$, the height of $\lp{\lambda/\nu}\mid_{x=m-2}$. \\
Furthermore, since $\mu$ is an even partition, all the minimal-height (\emph{NOT ``minimum"!}) points of $\lp{\lambda/\mu}$ possess even heights. Recall that $x=m$ is the first lowest point of $\lp{\lambda/\mu}\mid_{0\le x\le\lambda_1}$, indicating that all the minimal-height points of $\lp{\lambda/\nu}\mid_{0\le x<m-2} = \lp{\lambda/\mu}\mid_{0\le x<m-2}$ are strongly higher than $y=h$ and thus weakly higher than $y=h+2$, the height of $\lp{\lambda/\nu}\mid_{x=m-2}$. \\
To sum up, we have shown that $x=m-2$ is the last lowest point of $\lp{\lambda/\nu}\mid_{0\le x\le\lambda_1}$. Now our claim above has been proved. \\
Now suppose that $\Phi_{d,\lambda}(\lambda/\mu^{(1)}) = \Phi_{d,\lambda}(\lambda/\mu^{(2)}) = \lambda/\nu$. We show that $\lambda/\mu^{(1)} = \lambda/\mu^{(2)}$ as means to deduce that $\Phi_{d,\lambda}$ is injective. Let $x=m_i$ be the first lowest point of $\lp{\lambda/\mu^{(i)}}\mid_{0 \le x \le \lambda_1}$ ($i=1,2$). Then the claim above indicates that both $x=m_i-2$ ($i=1,2$) are the last lowest point of $\lp{\lambda/\nu}\mid_{0 \le x \le \lambda_1}$ and hence $m_1 = m_2$. Therefore, since the same ``turning upwards" operations (see Remark~\ref{rmk:turning}) on $\lp{\lambda/\mu^{(1)}}$ and $\lp{\lambda/\mu^{(2)}}$ generate the same result $\lp{\lambda/\nu}$, we have that $\lp{\lambda/\mu^{(1)}} = \lp{\lambda/\mu^{(2)}}$ and thus $\lambda/\mu^{(1)} = \lambda/\mu^{(2)}$ (because $\lp{\lambda/\mu}$ tells us which columns of $\lambda$ intersect the horizontal stripe $\lambda/\mu$, providing enough information to reconstruct $\lambda/\mu$). As a result, $\Phi_{d,\lambda}$ is injective.

\textbf{Surjectivity:} Given $\lambda/\nu\in\hori{d-1,\lambda}$. Note that $\lambda_1\le n-2d+a$ and $\lp{\lambda/\nu}$ goes downwards after the $\lambda_1$-th step. Therefore, let $H$ be the height of $\lp{\lambda/\nu}\mid_{x=\lambda_1}$. Then \begin{align*}
    H&\ge\lp{\lambda/\nu}\mid_{x=n-2d+a} =\lvert\lambda/\nu\rvert-(n-2d+a-\lvert\lambda/\nu\rvert) \\ &= n-2d+2-(n-2d+a-(n-2d+2)) = n -2d -a +4 \\&\ge n - 2\cdot\frac{n-a}{2} - a + 4 =4>0.
    \end{align*} However, the starting point $(0,0)$ of $\lp{\lambda/\nu}$ is of height $0$, indicating that $x=\lambda_1$ is not the lowest point of $\lp{\lambda/\nu}\mid_{0\le x\le\lambda_1}$. Now let $x=m^{\prime}<\lambda_1$ be the last lowest point of $\lp{\lambda/\nu}\mid_{0\le x\le\lambda_1}$. We further deduce that $m^\prime\le\lambda_1-4$ from the inequality $H\ge 4$ shown above. Since $x=m^\prime$ is the last lowest point of $\lp{\lambda/\nu}\mid_{0\le x\le\lambda_1}$, both the $(m^\prime+1)$-th and $(m^\prime+2)$-th steps of $\lp{\lambda/\nu}\mid_{0\le x\le\lambda_1}$ must be NE steps. Turn these two steps downwards, obtaining a new lattice path $\lp{\lambda/\mu}$ for some horizontal stripe $\lambda/\mu\in\hori{d,\lambda}\setminus\horipositive{d,\lambda}$ (we can reconstruct $\lambda/\mu$ from $\lambda$ and $\lp{\lambda/\mu}$, and $\lambda/\mu\notin\horipositive{d,\lambda}$ since the lowest point of $\lp{\lambda/\nu}$ must be weakly lower than the origin). Then Remark~\ref{rmk:turning} indicates that $\Phi_{d,\lambda}(\lambda/\mu) = \lambda/\nu$. In conclusion, $\Phi_{d,\lambda}$ is surjective.
\end{proof}

We are ready to state and prove our first positive combinatorial formula for $\grFrob(R(\MMM_{n,a});q)$.

\begin{proposition}\label{prop:bad-ver-formula}
    For integers $n,a,d$ such that $n>0$, $0\le a \le n$, $a\equiv n\mod{2}$, $0<d\le\frac{n-a}{2}$, we have that
    \[\Frob(R(\MMM_{n,a})_d) = \sum_{\lambda/\mu} s_\lambda\] summing over all the horizontal stripes
    \[\lambda/\mu\in\bigsqcup_{\substack{\lambda\vdash n \\ \lambda_1 \le n-2d+a}}\horipositive{d,\lambda}\]
    where $\horipositive{d,\lambda}$ is given by Equation~\eqref{eq:def-positive-horizontal-stripe-set}.
\end{proposition}

\begin{proof}
    Theorem~\ref{thm:conjugacy-module-character} indicates that
    \[\Frob(R(\MMM_{n,a})_d) = \{
            h_d[h_2] \cdot h_{n-2d} - h_{d-1}[h_2] \cdot h_{n-2d+2}
        \}_{\lambda_1 \leq n-2d+a}\] where we interpret $h_{-1}=0$.
    Therefore, the $d=0$ case is easily verified. For $0<d\le\frac{n-a}{2}$, applying Pieri's rule to the formula above yields that
    \begin{align*}
        \Frob(R(\MMM_{n,a})_d) = \sum_{\substack{\lambda\vdash n \\ \lambda_1\le n-2d+a}}\sum_{\lambda/\mu\in\hori{d,\lambda}}s_\lambda - \sum_{\substack{\lambda\vdash n \\ \lambda_1\le n-2d+a}}\sum_{\lambda/\mu\in\hori{d-1,\lambda}}s_\lambda
    \end{align*}
    Then Lemma~\ref{lem:map-NHS-to-HS-bijective} helps us replace the index set of the last summation with $\hori{d,\lambda}\setminus\horipositive{d,\lambda}$, indicating that
    \begin{align*}
        \Frob(R(\MMM_{n,a})_d) &= \sum_{\substack{\lambda\vdash n \\ \lambda_1\le n-2d+a}}\sum_{\lambda/\mu\in\hori{d,\lambda}}s_\lambda - \sum_{\substack{\lambda\vdash n \\ \lambda_1\le n-2d+a}}\sum_{\lambda/\mu\in\hori{d,\lambda}\setminus\horipositive{d,\lambda}}s_\lambda \\
        &= \sum_{\substack{\lambda\vdash n \\ \lambda_1\le n-2d+a}}\sum_{\lambda/\mu\in\horipositive{d,\lambda}}s_\lambda
    \end{align*}
    which completes the proof.
\end{proof}

\section{Good version: a positive combinatorial formula as an explicit graded refinement of $h_{(n-a)/2}[h_2]\cdot h_a$}\label{sec:good-ver}
In this section, we will prove Theorem~\ref{thm:main} using Proposition~\ref{prop:bad-ver-formula}. Note that Theorem~\ref{thm:main} is conciser than Proposition~\ref{prop:bad-ver-formula}, providing an explicit graded refinement of $\Frob(\CC[\MMM_{n,a}]) = \Frob(R(\MMM_{n,a})) = h_{(n-a)/2}[h_2]\cdot h_a$ in the form of a positive combinatorial formula.

In order to convert Proposition~\ref{prop:bad-ver-formula} into Theorem~\ref{thm:main}, we need to construct bijections between their index sets. We introduce some basic settings and notations as follows.

\emph{Fix} two integers $n$ and $a$ such that $n>0$, $0\le a\le n$ and $a\equiv n \mod{2}$. Let $0\le d\le\frac{n-a}{2}$ be an integer and $\lambda\vdash n$ be a partition such that $\lambda_1\le n-2d+a$. Then keep the notations in Section~\ref{sec:bad-ver}. In addition, define the set (recall Definition~\ref{def:width} for the definition of the statistic width $\wid{}$)
\begin{align}\label{eq:def-of-good-ver-set}
    \horiwid{d,\lambda} \coloneqq \{\lambda/\mu\in\hori{(n-a)/2,\lambda}\,:\,\wid{\lambda/\mu} = n-2d+a\}.
\end{align}

Then define the \emph{left shadow map}
\begin{align}\label{eq:left-shadow-map}
    \leftshadow{d,\lambda}\,:\,\horipositive{d,\lambda} \longrightarrow \horiwid{d,\lambda}
\end{align}
through the following steps in order:
\begin{itemize}
    \item Given $\lambda/\mu\in\horipositive{d,\lambda}$, construct a set \[S=\{i\in[\lambda_1]\,:\,\text{$(i,j)$ is a reflection pair of $\lp{\lambda/\mu}$ for some $j\in[n-2d+a]$}\}.\]
    (Recall the term ``reflection pair" given by Definition~\ref{def:pair}.)
    \item Let $\lambda/\nu$ be the unique horizontal stripe such that all the columns of $\lambda$ intersecting $\lambda/\nu$ are exactly indexed by the set $S$.
    \item Let $\leftshadow{d,\lambda}(\lambda/\mu)\coloneqq\lambda/\nu$.
\end{itemize}

\begin{lemma}\label{lem:left-shadow-map-well-def}
    $\leftshadow{d,\lambda}$ is well-defined, i.e. $\lambda/\nu\in\horiwid{d,\lambda}$.
\end{lemma}
\begin{proof}
    First, we show that $\lambda/\nu\in\hori{(n-a)/2,\lambda}$. Note that there are exactly $n-2d+a-\lvert\lambda/\mu\rvert =n-2d+a-(n-2d)=a$ SE steps in $\lp{\lambda/\mu}\mid_{0\le x\le n-2d+a}$. Recall that $\lambda/\mu\in\horipositive{d,\lambda}$ indicates that $\lp{\lambda/\mu}\mid_{0\le x\le \lambda_1}$ is weakly higher than the $x$-axis. Moreover, the height of $\lp{\lambda/\mu}\mid_{x=n-2d+a}$ is $\lvert\lambda/\mu\rvert-(n-2d+a-\lvert\lambda/\mu\rvert) = n-2d-a \ge n-2\cdot\frac{n-a}{2}-a = 0$, and $\lp{\lambda/\mu}\mid_{\lambda_1\le x\le n-2d+a}$ only contains SE steps. Consequently, $\lp{\lambda/\mu}\mid_{0\le x\le n-2d+a}$ is weakly higher than the $x$-axis. It follows that each horizontal leftward ray starting at all of the $a$ SE steps of $\lp{\lambda/\mu}\mid_{0\le x\le n-2d+a}$ must touch an NE step of $\lp{\lambda/\mu}\mid_{0\le x\le \lambda_1}$. As a result, $\lvert S \rvert = a$ and thus $\lvert\lambda/\nu\rvert = a$, indicating that $\nu\vdash n-a$. Furthermore, $\mu$ is an even partition, revealing that all the minimal-height (NOT ``minimum") points of $\lp{\lambda/\mu}$ must have even heights; hence, $\nu$ is also an even partition. Therefore, $\lambda/\nu\in\hori{(n-a)/2,\lambda}$. 

    It remains to show that $\wid{\lambda/\nu} = n - 2 d + a$. Note that we can intuitively construct $\lp{\lambda/\nu}$ according to $\lp{\lambda/\mu}$ sequentially as follows:
    \begin{itemize}
        \item Put a large light source on the right of $\lp{\lambda/\mu}\mid_{0\le x\le n-2d+a}$. It emits horizontal light rays leftwards. Then the heights of all the NE steps touched by these rays must be connected, and their minimum is $0$. Other steps perfectly match each other, forming reflection pairs. 

        \def\sequence{1,-1,1,1,1,-1,1,-1,-1,-1,1,1,1,1,-1,-1,1,1,1,1,1,-1,1,-1,-1,-1}
        \def\gridscale{0.5}
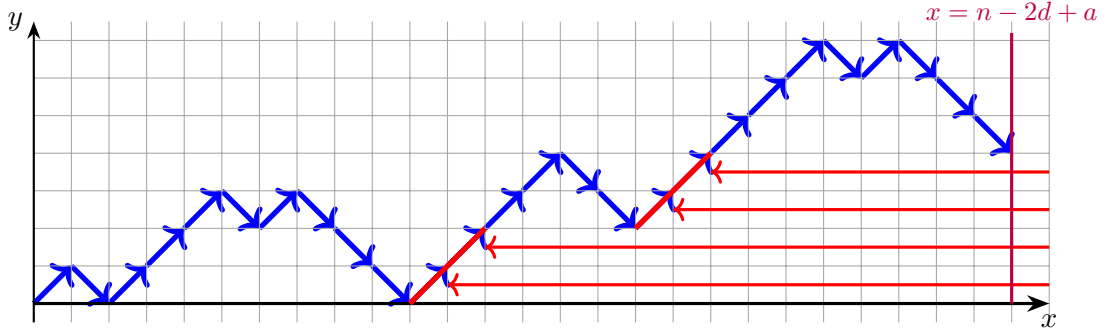
        \begin{tikzpicture}[scale=\gridscale]
            \tikzmath{
            integer \x,\y,\miny,\maxy;
            \x=0;\y=0;\miny=0;\maxy=0;
            for \s in \sequence {
            {\draw[->,line width=2pt,color = blue] (\x,\y) -- (\x+1,\y+\s);};
            \x=\x+1;\y=\y+\s;
            if \y>\maxy then {\maxy = \y;};
            if \y<\miny then {\miny = \y;};
            };
            {
            \draw[gray!70, thin] (0,\miny-0.5) grid (\x+1,\maxy+0.5);
            \draw[very thick, -Stealth] (0,0) -- (\x+1,0) node[below] {$x$};
            \draw[thick, -Stealth] (0,\miny-0.5) -- (0,\maxy+0.5) node[left]  {$y$};
            \draw[->,very thick,red] (27,0.5) -- (11,0.5);
            \draw[line width=2pt,red] (10,0) -- (12,2);
            \draw[->,very thick,red] (27,1.5) -- (12,1.5);
            \draw[->,very thick,red] (27,2.5) -- (17,2.5);
            \draw[line width=2pt,red] (16,2) -- (18,4);
            \draw[->,very thick,red] (27,3.5) -- (18,3.5);
            \draw[purple,very thick] (26,\miny) -- (26,\maxy+0.2) node[above] {\small $x=n-2d+a$};
            };
            }
        \end{tikzpicture}

    \item Whenever the light rays touch an NE step, replace it with an SE step, obtaining $\lp{\lambda/\nu}\mid_{0\le x\le n-2 d+a}$. Then the heights of these new SE steps are connected as well, and their maximum is $0$. In addition, these new SE steps (red in the figure below) and those original NE steps (red in the figure above) are symmetric with respect to the $x$-axis. Other steps perfectly match each other, forming reflection pairs. Therefore, \begin{align}\label{eq:wid-pre}n-2d+a\ge \max\{j\,:\,\text{$(i,j)$ is a reflection pair of $\lp{\lambda/\mu}$ for some $i$}\}.\end{align} Recall that $\lambda_1\le n-2d+a$. Then we have that
    \[n-2d+a \ge\max\{\lambda_1, \max\{j\,:\,\text{$(i,j)$ is a reflection pair of $\lp{\lambda/\mu}$ for some $i$}\}\}.\]
    Moreover, if $\lambda_1<n-2d+a$, then the last step of $\lp{\lambda/\mu}$ is an SE step and hence the equality of \eqref{eq:wid-pre} holds. To sum up, we deduce that
    \[n-2d+a =\max\{\lambda_1, \max\{j\,:\,\text{$(i,j)$ is a reflection pair of $\lp{\lambda/\mu}$ for some $i$}\}\},\] which means that $\wid{\lambda/\nu} = n-2d+a$ according to Definition~\ref{def:width}.
    
    \def\sequence{1,-1,1,1,1,-1,1,-1,-1,-1,-1,-1,1,1,-1,-1,-1,-1,1,1,1,-1,1,-1,-1,-1}
        \def\gridscale{0.5}
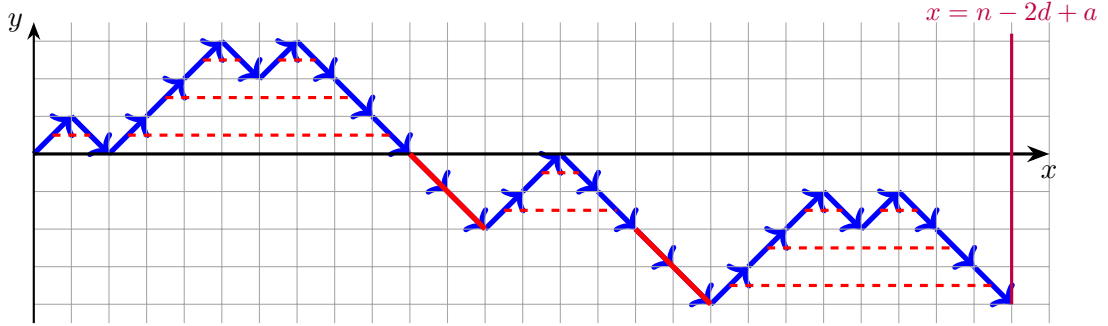
        \begin{tikzpicture}[scale=\gridscale]
            \tikzmath{
            integer \x,\y,\miny,\maxy;
            \x=0;\y=0;\miny=0;\maxy=0;
            for \s in \sequence {
            {\draw[->,line width=2pt,color = blue] (\x,\y) -- (\x+1,\y+\s);};
            \x=\x+1;\y=\y+\s;
            if \y>\maxy then {\maxy = \y;};
            if \y<\miny then {\miny = \y;};
            };
            {
            \draw[gray!70, thin] (0,\miny-0.5) grid (\x+1,\maxy+0.5);
            \draw[very thick, -Stealth] (0,0) -- (\x+1,0) node[below] {$x$};
            \draw[thick, -Stealth] (0,\miny-0.5) -- (0,\maxy+0.5) node[left]  {$y$};
            \draw[line width=2pt,red] (10,0) -- (12,-2);
            \draw[line width=2pt,red] (16,-2) -- (18,-4);
            \draw[purple,very thick] (26,\miny) -- (26,\maxy+0.2) node[above] {\small $x=n-2d+a$};
            \draw[red,dashed,very thick] (0.5,0.5) -- (1.5,0.5);
            \draw[red,dashed,very thick] (2.5,0.5) -- (9.5,0.5);
            \draw[red,dashed,very thick] (3.5,1.5) -- (8.5,1.5);
            \draw[red,dashed,very thick] (4.5,2.5) -- (5.5,2.5);
            \draw[red,dashed,very thick] (6.5,2.5) -- (7.5,2.5);
            \draw[red,dashed,very thick] (12.5,-1.5) -- (15.5,-1.5);
            \draw[red,dashed,very thick] (13.5,-0.5) -- (14.5,-0.5);
            \draw[red,dashed,very thick] (18.5,-3.5) -- (25.5,-3.5);
            \draw[red,dashed,very thick] (19.5,-2.5) -- (24.5,-2.5);
            \draw[red,dashed,very thick] (20.5,-1.5) -- (21.5,-1.5);
            \draw[red,dashed,very thick] (22.5,-1.5) -- (23.5,-1.5);
            };
            }
        \end{tikzpicture}
    \end{itemize}

    In one word, we have shown that $\lambda/\nu\in\hori{(n-a)/2,\lambda}$ and $\wid{\lambda/\nu} = n-2d+a$, indicating that $\lambda/\nu\in\horiwid{d,\lambda}$.
\end{proof}

Now we want to construct the inverse map of $\leftshadow{d,\lambda}$ given by \eqref{eq:left-shadow-map}, called the \emph{right shadow map}.

\begin{align}\label{eq:def-of-right-shadow-map}
    \rightshadow{d,\lambda}\,:\,\horiwid{d,\lambda} \longrightarrow \horipositive{d,\lambda}.
\end{align}

Given $\lambda/\nu\in\horiwid{d,\lambda}$, we associate $\rightshadow{d,\lambda}(\lambda/\nu)$ using the following steps:
\begin{itemize}
    \item Let $T=\{j\in\mathbb{N}\,:\,\text{$(i,j)$ is a reflection pair for some $i$}\}$.
    \item $\lambda/\nu\in\horiwid{d,\lambda}$ indicates that $\wid{\lambda/\nu} = n-2d+a$, so we have that $\max T \le n-2d+a$. Let $\lambda/\mu$ be the unique horizontal stripe such that all the columns of $\lambda$ intersecting $\lambda/\mu$ are exactly indexed by $[n-2d+a]\setminus T$.
    \item Let $\rightshadow{d,\lambda}(\lambda/\nu)\coloneqq\lambda/\mu$.
\end{itemize}

\begin{lemma}\label{lem:right-shadow-map-well-def}
    $\rightshadow{d,\lambda}$ is well-defined, i.e. $\lambda/\mu\in\horipositive{d,\lambda}$.
\end{lemma}
\begin{proof}
    Initially, we show that $\mu\vdash 2d$. Since $\nu\vdash\frac{n-a}{2}$ indicates that $\lvert\lambda/\nu\rvert = a$, $\lp{\lambda/\nu}$ totally has $a$ NE steps and hence $\lvert T \rvert = a$. Therefore, $\lvert\lambda/\mu\rvert = \lvert [n-2d+a]\setminus T \rvert =n-2d+a-a=n-2d$, revealing that $\mu\vdash 2d$. 
    
    It remains to show that $\lp{\lambda/\mu}\mid_{0 \le x \le n-2d+a}$ is weakly higher than the $x$-axis. Similarly to the proof of Lemma~\ref{lem:left-shadow-map-well-def}, we describe how to construct $\lp{\lambda/\mu}$ according to $\lp{\lambda/\nu}$ intuitively. The construction needs two steps:
    \begin{itemize}
        \item Put a large light source on the left of $\lp{\lambda/\nu}$ and let it emit horizontal light rays rightwards. Then the heights of all the SE steps touched by these rays must be connected, and their maximum is $0$.

        \def\sequence{1,-1,1,1,1,-1,1,-1,-1,-1,-1,-1,1,1,-1,-1,-1,-1,1,1,1,-1,1,-1,-1,-1}
        \def\gridscale{0.5}
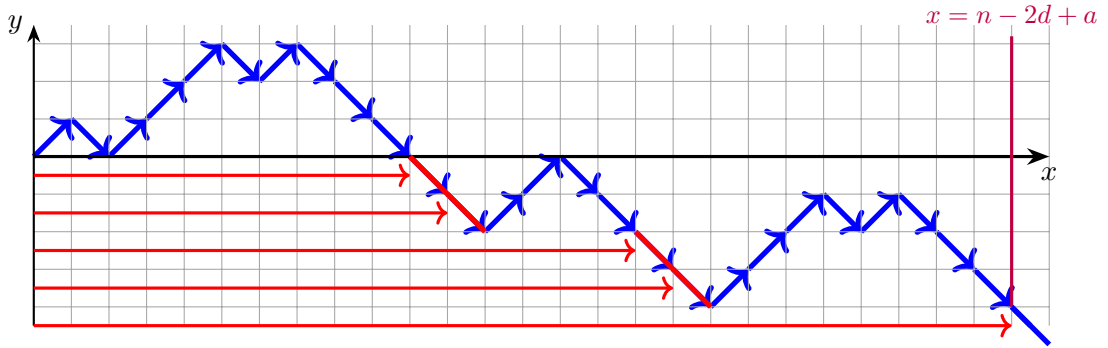
        \begin{tikzpicture}[scale=\gridscale]
            \tikzmath{
            integer \x,\y,\miny,\maxy;
            \x=0;\y=0;\miny=0;\maxy=0;
            for \s in \sequence {
            {\draw[->,line width=2pt,color = blue] (\x,\y) -- (\x+1,\y+\s);};
            \x=\x+1;\y=\y+\s;
            if \y>\maxy then {\maxy = \y;};
            if \y<\miny then {\miny = \y;};
            };
            {
            \draw[line width=2pt,color = blue] (\x,\y) -- (\x+0.5,\y-0.5);
            \draw[gray!70, thin] (0,\miny-0.5) grid (\x+1,\maxy+0.5);
            \draw[very thick, -Stealth] (0,0) -- (\x+1,0) node[below] {$x$};
            \draw[thick, -Stealth] (0,\miny-0.5) -- (0,\maxy+0.5) node[left]  {$y$};
            \draw[line width=2pt,red] (10,0) -- (12,-2);
            \draw[line width=2pt,red] (16,-2) -- (18,-4);
            \draw[purple,very thick] (26,\miny) -- (26,\maxy+0.2) node[above] {\small $x=n-2d+a$};
            \draw[->,very thick,red] (0,-0.5) -- (10,-0.5);
            \draw[->,very thick,red] (0,-1.5) -- (11,-1.5);
            \draw[->,very thick,red] (0,-2.5) -- (16,-2.5);
            \draw[->,very thick,red] (0,-3.5) -- (17,-3.5);
            \draw[->,very thick,red] (0,-4.5) -- (26,-4.5);
            \draw[line width = 2pt,blue] (26,-4) -- (27,-5);
            };
            }
        \end{tikzpicture}

        \item Replace all the SE steps of $\lp{\lambda/\nu}\mid_{0\le x\le n-2d+a}$ touched by light rays with NE steps to obtain $\lp{\lambda/\mu}$. Clearly, these new NE steps (red in the figure below) and old SE steps (red in the figure above) are symmetric with respect to the $x$-axis, so these new NE steps are weakly higher than the $x$-axis, indicating that $\lp{\lambda/\mu}\mid_{0 \le x \le n-2d+a}$ is weakly higher than the $x$-axis, which completes the proof.

        \def\sequence{1,-1,1,1,1,-1,1,-1,-1,-1,1,1,1,1,-1,-1,1,1,1,1,1,-1,1,-1,-1,-1}
        \def\gridscale{0.5}
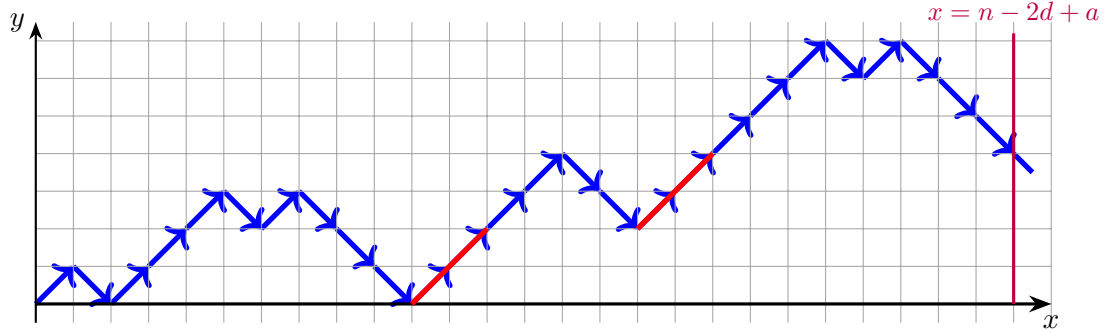
        \begin{tikzpicture}[scale=\gridscale]
            \tikzmath{
            integer \x,\y,\miny,\maxy;
            \x=0;\y=0;\miny=0;\maxy=0;
            for \s in \sequence {
            {\draw[->,line width=2pt,color = blue] (\x,\y) -- (\x+1,\y+\s);};
            \x=\x+1;\y=\y+\s;
            if \y>\maxy then {\maxy = \y;};
            if \y<\miny then {\miny = \y;};
            };
            {
            \draw[line width=2pt,color = blue] (\x,\y) -- (\x+0.5,\y-0.5);
            \draw[gray!70, thin] (0,\miny-0.5) grid (\x+1,\maxy+0.5);
            \draw[very thick, -Stealth] (0,0) -- (\x+1,0) node[below] {$x$};
            \draw[thick, -Stealth] (0,\miny-0.5) -- (0,\maxy+0.5) node[left]  {$y$};
            \draw[line width=2pt,red] (10,0) -- (12,2);
            \draw[line width=2pt,red] (16,2) -- (18,4);
            \draw[purple,very thick] (26,\miny) -- (26,\maxy+0.2) node[above] {\small $x=n-2d+a$};
            };
            }
        \end{tikzpicture} 
    \end{itemize}
\end{proof}

Now we mention the last technical result before the proof of Theorem~\ref{thm:main}.

\begin{lemma}\label{lem:inverse-map}
    $\leftshadow{d,\lambda}$ and $\rightshadow{d,\lambda}$ are the inverse map of each other.
\end{lemma}
\begin{proof}
    The proofs of Lemma~\ref{lem:left-shadow-map-well-def} and Lemma~\ref{lem:right-shadow-map-well-def} respectively illustrate how to convert $\lp{\lambda/\mu}$ into $\lp{\leftshadow{d,\lambda}(\lambda/\mu)}$ and how to convert $\lp{\lambda/\nu}$ into $\lp{\rightshadow{d,\lambda}(\lambda/\nu)}$ using ``horizontal rays" intuitively. Note that these two operations on lattice paths are the inverse of each other. Furthermore, we can reconstruct a horizontal stripe using the lattice path associated with it. Therefore, $\leftshadow{d,\lambda}$ and $\rightshadow{d,\lambda}$ are the inverse of each other.
\end{proof}

Finally, we are ready to prove Theorem~\ref{thm:main}.

\begin{proof}[Proof of Theorem~\ref{thm:main}]
    By Proposition~\ref{prop:bad-ver-formula}, we have that
    \[\Frob(R(\MMM_{n,a})_d) = \sum_{\substack{\lambda\vdash n \\ \lambda_1\le n-2d+a}}\sum_{\lambda/\mu\in\horipositive{d,\lambda}} s_\lambda.\]
    The one-to-one correspondence between $\horipositive{d,\lambda}$ and $\horiwid{d,\lambda}$ arising from Lemma~\ref{lem:inverse-map} enables us to replace the index set of the last summation above with $\horiwid{d,\lambda}$. That is,
    \[\Frob(R(\MMM_{n,a})_d) = \sum_{\substack{\lambda\vdash n \\ \lambda_1 \le n-2d+a}}\sum_{\lambda/\mu\in\horiwid{d,\lambda}}s_\lambda,\]
    which is equivalent to Theorem~\ref{thm:main}. (Recall the definition of $\horiwid{d,\lambda}$ in Equation~\eqref{eq:def-of-good-ver-set}.)
\end{proof}

\section{Application}\label{sec:appli}

Recall that Proposition~\ref{prop:bad-ver-formula} assigns a restriction $\lambda_1\le n-2d+a$ to the index range of the summation. In fact, we can remove this restriction in some cases.

\begin{corollary}\label{cor:bad-ver-remove-len-res}
    For integers $n,a,d$ in Proposition~\ref{prop:bad-ver-formula}, if further $d\le a$, then
    \[\Frob(R(\MMM_{n,a})_d) = \sum_{\lambda/\mu}s_\lambda\]
    where $\lambda/\mu$ ranges over $\bigsqcup_{\lambda \vdash n}\horipositive{d,\lambda}$.
\end{corollary}
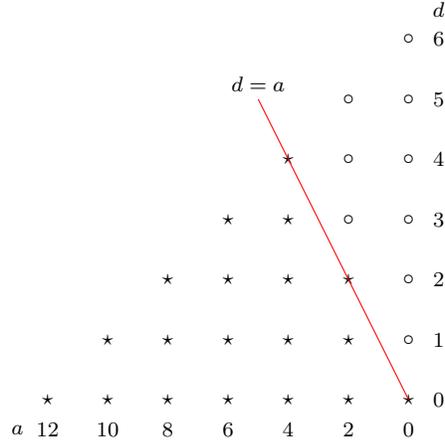
\begin{proof}
    By concentrating on the summation index sets of Proposition~\ref{prop:bad-ver-formula} and Corollary~\ref{cor:bad-ver-remove-len-res}, it suffices to show that $\horipositive{d,\lambda} = \varnothing$ whenever $\lambda_1>n-2d+a$. Assume, for the sake of contradiction, that $\lambda/\mu\in\horipositive{d,\lambda}$. Among columns $1,2,\cdots,\mu_1$, write $c$ for the number of columns intersecting $\lambda/\mu$ and $\mu_1-c$ for the number of columns not intersecting $\lambda/\mu$. Then $c\ge\mu_1-c$ since $\lp{\lambda/\mu}\mid_{0\le x\le\mu_1}$ is weakly higher than the $x$-axis. Consequently, we deduce that $c\ge\frac{\mu_1}{2}$ and hence 
    \begin{align*}
        \lambda_1 &= \mu_1 + |\{\text{cells of $\lambda/\mu$ not under $\mu$}\}| =\mu_1 + |\lambda/\mu|-c = \mu_1 + n-2d -c \\
        &\le \mu_1 + n -2d - \frac{\mu_1}{2} = n-2d + \frac{\mu_1}{2} \le n-2d +d \le n-2d+a 
    \end{align*}
    where the second-to-last sign of inequality arises from $\mu\vdash 2d$. However, the inequality above contradicts $\lambda_1>n-2d+a$.
\end{proof}

Note that the summation in Corollary~\ref{cor:bad-ver-remove-len-res} does not depend on $a$. We thus immediately obtain some $\symm_n$-equivariant isomorphisms of the form $R(\MMM_{n,a})_d \cong R(\MMM_{n,a^\prime})_d$ with $a\neq a^\prime$ as follows.

\begin{corollary}\label{cor:isom}
    For nonnegative integers $n,a,a^\prime,d$ such that $n>0$, $\max\{a,a^\prime\}\le n$, $a\equiv a^\prime\equiv n \mod{2}$, $d\le\min\{a,a^\prime,\frac{n-a}{2},\frac{n-a^\prime}{2}\}$, we have isomorphisms of $\symm_n$-modules
    \[R(\MMM_{n,a})_d\cong R(\MMM_{n,a^\prime})_d.\]
\end{corollary}

\begin{remark}\label{rmk:isom}
    Theorem~\ref{thm:conjugacy-module-character} indicates chains of $\symm_n$-module surjections where $\delta=n\mod{2} \in \{0,1\}$ (\cite[Lemma 5.16]{liu2025involution})
    \begin{scriptsize}
      \begin{displaymath} \displaystyle \scalebox{1}{
        \xymatrix{
        &&&&&R_n(\MMM_{n,\delta})_{\frac{n-\delta}{2}}\\
        &&&&R(\MMM_{n,\delta+2})_{\frac{n-\delta-2}{2}}\ar@{->>}[r] &R(\MMM_{n,\delta})_{\frac{n-\delta-2}{2}}\\
        && & R(\MMM_{n,\delta+4})_{\frac{n-\delta-4}{2}}
        \ar@{->>}[r] &R(\MMM_{n,\delta+2})_{\frac{n-\delta-4}{2}}\ar@{->>}[r] &R(\MMM_{n,\delta})_{\frac{n-\delta-4}{2}}\\
        & &\begin{sideways}$\ddots$\end{sideways} &\vdots &\vdots &\vdots \\
         &R(\MMM_{n,n-2})_{1}\ar@{->>}[r] &\cdots\ar@{->>}[r] &R(\MMM_{n,\delta+4})_{1}\ar@{->>}[r] &R(\MMM_{n,\delta+2})_{ 1}\ar@{->>}[r] &R(\MMM_{n,\delta})_{1}\\
        R(\MMM_{n,n})_{0}\ar@{->>}[r] &R(\MMM_{n,n-2})_{0}\ar@{->>}[r] &\cdots\ar@{->>}[r] &R(\MMM_{n,\delta+4})_{0}\ar@{->>}[r] &R(\MMM_{n,\delta+2})_{0}\ar@{->>}[r] &R(\MMM_{n,\delta})_{0}
        }}
      \end{displaymath}
      \end{scriptsize}

    Interestingly, Corollary~\ref{cor:isom} replaces all the surjections ``$\twoheadrightarrow$" between $\symm_n$-modules $R(\MMM_{n,a})_d$ weakly under the straight line $d=a$ with isomorphisms ``$\cong$". That is, if we use circles and stars to represent each graded component $R(\MMM_{n,a})_d$ above and thus generate the following diagram, then all the stars on the same row are isomorphic. In this example, $n=12$.
    \begin{center}
        \begin{scriptsize}
            \begin{tikzpicture}[scale = 0.4]
                \node at (20,0) {$\star$};

                \node at (22,0) {$\star$};
                \node at (22,2) {$\star$};

                \node at (24,0) {$\star$};
                \node at (24,2) {$\star$};
                \node at (24,4) {$\star$};

                \node at (26,0) {$\star$};
                \node at (26,2) {$\star$};
                \node at (26,4) {$\star$};
                \node at (26,6) {$\star$};

                \node at (28,0) {$\star$};
                \node at (28,2) {$\star$};
                \node at (28,4) {$\star$};
                \node at (28,6) {$\star$};
                \node at (28,8) {$\star$};

                \node at (30,0) {$\star$};
                \node at (30,2) {$\star$};
                \node at (30,4) {$\star$};
                \node at (30,6) {$\circ$};
                \node at (30,8) {$\circ$};
                \node at (30,10) {$\circ$};
                
                \node at (32,0) {$\star$};
                \node at (32,2) {$\circ$};
                \node at (32,4) {$\circ$};
                \node at (32,6) {$\circ$};
                \node at (32,8) {$\circ$};
                \node at (32,10) {$\circ$};
                \node at (32,12) {$\circ$};

                \node at (19,-1) {$a$};

                \node at (20,-1) {$12$};
                \node at (22,-1) {$10$};
                \node at (24,-1) {$8$};
                \node at (26,-1) {$6$};
                \node at (28,-1) {$4$};
                \node at (30,-1) {$2$};
                \node at (32,-1) {$0$};

                \node at (33,13) {$d$};
                \node at (33,12) {$6$};
                \node at (33,10) {$5$};
                \node at (33,8) {$4$};
                \node at (33,6) {$3$};
                \node at (33,4) {$2$};
                \node at (33,2) {$1$};
                \node at (33,0) {$0$};
                
                \draw[red] (32,0) -- (27,10);
                \node at (27,10.5) {$d=a$};
            \end{tikzpicture}
        \end{scriptsize}
    \end{center}
\end{remark}

\section{Future directions}\label{sec:future}

Liu et al. \cite{liu2025involution} left some interesting open problems. Our positive combinatorial formulae may yield potential ways to address two of them.

\begin{problem}\label{prob:stat}
    Find a statistic $\mathrm{stat}:\MMM_{n,a}\rightarrow\mathbb{Z}_{\ge 0}$ such that
    \[\Hilb(R(\MMM_{n,a});q) = \sum_{w\in\MMM_{n,a}}q^{\mathrm{stat}(w)}.\]
\end{problem}

Here are some ideas about Problem~\ref{prob:stat}. Let $\mathrm{IND}$ be the index set of the summation in Theorem~\ref{thm:main}. Then define $\mathrm{DIM} \coloneqq \{(P,\lambda/\mu)\,:\, \text{$\lambda/\mu\in\mathrm{IND}$ and $P\in\SYT(\lambda)$}\}$ to interpret the dimension of each graded component $R(\MMM_{n,a})_d$ of $R(\MMM_{n,a})$. That is, $\mathrm{DIM} = \bigsqcup_{d=0}^{\frac{n-a}{2}}\mathrm{DIM}_d$ where $\mathrm{DIM}_d\coloneqq\{(P,\lambda/\mu)\in\mathrm{DIM}\,:\,\wid{\lambda/\mu} = n-2d+a\}$, and Theorem~\ref{thm:main} indicates that $\lvert\mathrm{DIM}_d\rvert = \dim(R(\MMM_{n,a})_d)$. It suffices to construct a bijection $\MMM_{n,a}\overset{1:1}{\rightarrow}\mathrm{DIM}$, and then find a statistic $\mathrm{stat}$ such that $\mathrm{stat}(w) = d$ for all $w$ mapped to $\mathrm{DIM}_d$ ($d=0,1,\cdots,\frac{n-a}{2}$).

Such a bijection is easily constructible. Given $w\in\MMM_{n,a}$, we replace all the diagonal entries of the permutation matrix of $w$ with $0$, yielding a matrix $M$. Note that $M$ is a symmetric $0$-$1$ matrix with totally $(n-a)$ $1$'s and without diagonal $1$'s. Applying Proposition~\ref{prop:sym-rsk-odd}, it follows that $\mathrm{RSK}(M) = (P,P)$ where $\mathrm{sh}(P)^\prime=\mu\vdash (n-a)$ is an even partition (here we write $\nu^\prime$ for the conjugate of a partition $\nu$). Now use row insertion operation to insert all the $a$ integers in $[n]$ not appearing in $P$ into $P$ (in increasing order), yielding a new standard Young tableau $Q$ of shape $\lambda$. Consequently, we obtain a pair $(Q,\lambda/\mu)\in\mathrm{DIM}$. The tableau version of Pieri's rule (Proposition~\ref{prop:tab-pieri}) indicates that the map $w\mapsto (Q,\lambda/\mu)$ above gives a bijection $\MMM_{n,a}\overset{1:1}{\rightarrow}\mathrm{DIM}$.

However, given $w\in\MMM_{n,a}$, it is difficult to figure out the graded part $\mathrm{DIM}_d$ that the involution $w\in\MMM_{n,a}$ is mapped to. Therefore, the statistic $\mathrm{stat}$ remains mysterious.

\begin{problem}\label{prob:basis}
    Find an explicit monomial basis of $R(\MMM_{n,a})$.
\end{problem}

Here are some ideas for Problem~\ref{prob:basis}. Like what we did for Problem~\ref{prob:stat}, we use the row insertion operation. Proposition~\ref{prop:bad-ver-formula} plays a crucial role now. Fix $0\le d\le\frac{n-a}{2}$ and consider the index set in Proposition~\ref{prop:bad-ver-formula}. Given a horizontal stripe \[\lambda/\mu\in\bigsqcup_{\substack{\lambda\vdash n \\ \lambda_1\le n-2d+a}}\horipositive{d,\lambda}\]
and $P\in\SYT(\lambda)$, we want to convert the pair $(P,\lambda/\mu)$ into a monomial $m\in\CC[\xxx_{n\times n}]$ of degree $d$ so that we can construct a basis of $R(\MMM_{n,a})_d$ using all of such monomials $m$.

The inverse of the row insertion operation enables us to push the horizontal stripe $\lambda/\mu$ out of $P$, yielding a new tableau $Q$ of shape $\mu$. Let $M=\mathrm{RSK}^{-1}(Q^\prime,Q^\prime)$. Since $\mathrm{sh}(Q)=\mu\vdash 2d$ is an even partition, $M$ must be a symmetric $0$-$1$ matrix with totally $2d$ $1$'s on distinct rows and columns and avoiding the diagonal (see Propositions~\ref{prop:sym-rsk} and \ref{prop:sym-rsk-odd}). Write \[m = \prod_{\substack{i<j\\M_{i,j}=1}}x_{i,j}.\] The tableau version of Pieri's rule (Proposition~\ref{prop:tab-pieri}) indicates that the map $(P,\lambda/\mu)\mapsto m$ is injective. Therefore, we exactly obtain $\dim(R(\MMM_{n,a})_d)$ distinct monomials $m$, so we guess that such monomials $m$ form a basis of $R(\MMM_{n,a})_d$. This conjecture has been verified for $a=0$ and $n\le 8$ through computer coding.

\section{Acknowledgements}\label{sec:acknowledgements}

We thank Professor Brendon Rhoades for initiating the fruitful project about involution matrix loci and orbit harmonics, which also presents interesting open problems. In addition, we thank the referee of the journal Mathematische Zeitschrift who asked for a positive combinatorial formula without minus signs for $\grFrob(R(\MMM_{n,a});q)$.

\printbibliography

\end{document}